\documentclass{amsart}

\usepackage{amsfonts}
\usepackage{amssymb}
\usepackage{amsthm}
\usepackage{eucal}
\usepackage{graphicx}
\usepackage{tikz}

\theoremstyle{plain}
\newtheorem{theorem}{Theorem}[section]
\newtheorem*{recalltheorem}{Theorem}
\newtheorem{lemma}[theorem]{Lemma}

\newtheorem{corollary}[theorem]{Corollary}

\theoremstyle{definition}

\theoremstyle{remark}

\renewcommand{\a}{\mathfrak{a}}
\renewcommand{\b}{\mathfrak{b}}
\newcommand{\m}{\mathfrak{m}}
\newcommand{\n}{\mathfrak{n}}
\newcommand{\G}{\Gamma}
\newcommand{\I}{\mathbf{I}}

\newcommand{\pideal}{\vartriangleleft}

\begin{document}
\title[intersection graph]{On cycles in intersection graphs of rings}
\author{A. Azimi}
\author{A. Erfanian}
\author{M. Farrokhi D. G.}
\author{N. Hoseini}
\address{Department of Pure Mathematics, Ferdowsi University of Mashhad, Mashhad, Iran}
\email{ali.azimi61@gmail.com}
\address{Department of Pure Mathematics, Ferdowsi University of Mashhad, Mashhad, Iran}
\email{erfanian@math.um.ac.ir}
\address{Department of Pure Mathematics, Ferdowsi University of Mashhad, Mashhad, Iran}
\email{m.farrokhi.d.g@gmail.com}
\address{Department of Pure Mathematics, Ferdowsi University of Mashhad, Mashhad, Iran}
\email{nesa.hoseini@gmail.com}
\begin{abstract}
Let $R$ be a commutative ring with non-zero identity. We describe all $C_3$- and $C_4$-free intersection graph of non-trivial ideals of $R$ as well as $C_n$-free intersection graph when $R$ is a reduced ring. Also, we shall describe all complete, regular and $n$-claw-free intersection graphs. Finally, we shall prove that almost all Artin rings $R$ have Hamiltonian intersection graphs. We show that such graphs are indeed pancyclic.
\end{abstract}
\maketitle
\section{Introduction}
If $S=\{S_1,\ldots,S_n\}$ is a family of sets, then the intersection graph of $S$, is the graph having $S$ as its vertex set with $S_i$ adjacent to $S_j$ if $i\neq j$ and $S_i\cap S_j\neq\emptyset$. A well-know theorem due to Marczewski \cite{tam-frm} states that all graphs are intersection graph.

An interesting case of intersection graphs is when the members of $S$ have an algebraic structure. Bosak \cite{jb} was the first who studied graphs arising from semigroups. Cs\'{a}k\'{e}any and Poll\'{a}k \cite{bc-gp} defined and studied the intersection graphs of nontrivial proper subgroups of groups. Zelinka \cite{bz} continued the work of Cs\'{a}k\'{e}any and Poll\'{a}k on intersection graphs of subgroups of finite abelian groups, and later Shen \cite{rs} studies such graphs and classifies all finite groups whose intersection graphs of nontrivial subgroups are disconnected. Herzog, Longobardi and Maj \cite{mh-pl-mm} study the intersection graphs of maximal subgroups of finite groups and among other results classify all finite groups with disconnected graph. The same as for groups, the intersection graphs of ideals of rings and subspaces of vector spaces have been discussed in \cite{ic-sg-tkm-mks,shj-njr-1,shj-njr-2}.

Let $R$ be a commutative ring with a non-zero identity. The intersection graph of $R$, denoted by $\G(R)$, is a graph whose vertices are the nontrivial ideals of $R$ and two distinct vertices are joined by an edge if the corresponding ideals of $R$ have a non-zero intersection.

In this paper, we study the cycle structure of intersection graphs. First we classify all Artin rings with a regular (hence complete) intersection graph. Next we shall investigate all rings $R$ whose intersection graphs $\G(R)$ do not have an induced cycle of  length $3$ or $4$. Also, we show that if $R$ is a reduce ring, then $\G(R)$ is $C_n$-free $(n\geq5)$ if and only if $R$ has no ideal which is the direct sum of $n$ non-zero ideals. The same result is also established for $n$-claws instead of $n$-cycles. In the last section, we shall prove that except few cases all other Artin rings have Hamiltonian intersection graphs. Using simple modifications of the given Hamiltonian cycle, we show that $\G(R)$ is pancyclic whenever it is Hamiltonian. Recall that an $n$-claw (a claw) is the star graph $K_{1,n}$ ($K_{1,3}$). Also a graph is called pancyclic if it contains cycles of possible arbitrary sizes $\geq3$.

The following theorem will be used without further reference.
\begin{recalltheorem}[{\cite[Theorem VI.2]{brm}}]
Let $R$ be an Artin commutative ring with a non-zero identity. Then 
\[R=R_1\oplus\cdots\oplus R_n,\]
where $R_1,\ldots,R_n$ are local rings.
\end{recalltheorem}
If $R$ is a ring, then the ideals $\a_1,\ldots,\a_n$ are called \textit{independent} if \[\a_i\cap(\a_1+\cdots+\a_{i-1}+\a_{i+1}+\cdots+\a_n)=0\]
for $i=1,\ldots,n$. In other words, $(\a_1,\ldots,\a_n)=\a_1\oplus\cdots\oplus\a_n$ is the direct sum of $\a_1,\ldots,\a_n$. All rings in this paper are commutative rings with a non-zero identity.
\section{$C_n$-free intersection graphs}
As a most simple property we may investigate on intersection graph $\G(R)$ of ideals of a ring, is whether $\G(R)$ is a complete graph. We show that the class of Artin rings with a complete intersection graph coincides with the class of Artin rings with a regular intersection graph and then characterize all such rings.
\begin{theorem}
Let $R$ be an Artin ring, which is not a direct sum of two fields. If $\G(R)$ is regular, then it is complete.
\end{theorem}
\begin{proof}
First we show that $R$ has no direct factor, which is a field. If $R=S\oplus F$, where $F$ is a field and $S$ is not a field, then $N_{\G(R)}(F)=\{\a\oplus F:0\neq\a\pideal S\}$ and $N_{\G(R)}(S)=\{\a,\a\oplus F:0\neq\a\pideal S\}$. Hence, $\deg_{\G(R)}S>\deg_{\G(R)}F$, which is a contradiction. Therefore, each maximal ideal of $R$ is adjacent to all other vertices of $\G(R)$, from which it follows that $\G(R)$ is a complete graph.
\end{proof}
\begin{theorem}
If $R$ is an Artin ring, then the graph $\G(R)$ is complete if and only if there exists a sequence of rings $R_1,\ldots,R_n$, in which $R=R_1$, $(R_i,R_{i+1})$ is a local ring for all $i=1,\ldots,n-1$ and $R_n$ is a field.
\end{theorem}
\begin{proof}
If $R$ is not a local ring, then $R=S\oplus T$ for some non-zero rings $S$ and $T$. But then $S\cap T=0$, which is a contradiction. Thus $R=(R,\m)$ is a local ring. Continuing this way for $\m$ instead of $R$ the result follows. The converse is obvious.
\end{proof}

In the following two theorems, we shall consider conditions under which the intersection graph of a ring is a star graph, which also results in a characterization of rings with a bipartite intersection graph.
\begin{theorem}\label{pendant}
Let $R$ be a ring, which is neither a direct sum of two fields nor a direct sum of a field with a local ring $(S,\m)$ such that $\m$ is a field. If $\G(R)$ has a pendant, then $\G(R)$ is a star graph.
\end{theorem}
\begin{proof}
Let $\a\in V(\G(R))$ be a pendant. If $\a$ is a maximal ideal, then it is easy to see that $R=(R,\a)$ is a local ring and $\a=(x)$ is a principal ideal. Let $\b$ be the ideal of $R$ adjacent to $\a$. Then $\b=(x^2)$ and $(x^3)=0$, hence $\G(R)$ is an edge.

If $\a$ is not a maximal ideal, then there exists a unique maximal ideal $\m$ of $R$ containing $\a$. Clearly, $\a=(x)$ is principal and $(x^2)=0$. If $R$ is not a local ring, then there exists a maximal ideal $\n$ such that $\a\cap\n=0$. Thus $R=\a\oplus\n$ and $\a$ is a field. Then $\m=\a+\b$ for some ideal $\b$ of $\n$. But then $(\n,\b)$ is a local ring such that $\b$ is a field, which is a contradiction. Therefore $(R,\m)$ is a local ring. Clearly, $\m=(x,y)$ for some $y\in R$. If $\m$ is principal, then we may assume that $\m=(y)$. Thus $(y^2)=(x)$ and $(y^3)=0$, which implies that $\G(R)$ is an edge. If $\m$ is not principal, then $(x)\cap(y)=0$ and consequently $xy=0$. Since $(x)\subseteq (x)+(y^2)\subseteq (x)+(y)=\m$, it follows that $(y^2)=0$. Hence $\m^2=0$ so that $\m$ is a vector space over the field $F=R/\m$, where the multiplication is defined by $(r+\m)\cdot m=rm$ for all $r\in R$ and $m\in\m$. Clearly, there is a one to one correspondence between ideals of $R$ contained in $\m$ and subspaces of $(\m,F)$. Hence $\dim_F\m=2$ so that $\G(R)$ is a star graph.
\end{proof}
\begin{theorem}\label{triangle}
\label{main}
 If $\G(R)$ is triangle-free, then $\G(R)$ is star or two isolated vertices.
\end{theorem}
\begin{proof}
If $R$ is not a local ring, then there exist two distinct maximal ideals $\m_1$ and $\m_2$ in $R$. Since $\G(R)$ is triangle-free, we should have $\m_1\cap\m_2=0$. Hence $R=\m_1\oplus\m_2$. Let $F_1=R/\m_1$ and $F_2=R/\m_2$. Then $R\cong F_1\oplus F_2$ and $\G(R)$ is the union of two isolated vertices.

Now, suppose that $(R,\m)$ is a local ring. We have two cases for $\m$.

Case 1: $\m$ is not a principal ideal. First we show that $\m^2=0$. If $x\in\m$ and $y\in \m\setminus xR$, then $ xR\cap yR=0$. Thus $xy=0$ so that $(\m\setminus xR)x=0$. On the other hand, if $r\in R$, then $xr=y+(xr- y)$ so that $xrx =0$. Thus $xRx=0$ and consequently $\m x=0$. Hence $\m^2=0$.

Let $F=R/\m$. The same as in the proof of Theorem \ref{pendant}, $(\m,F)$ is a vector space. If $\dim_F\m\geq3$ and $\lbrace x, y, z\rbrace$ is an independent set in $(\m,F)$, then the set of ideals $\{(x),(x,y),(x,y,z)\}$ induces a triangle in $\G(R)$, which is a contradiction. Thus $\dim_F\m\leq2$.

If $\dim_F\m=2$, then every two distinct non-trivial ideals of $R$ different from $\m$ are disjoint. Thus $\G(R)$ is a star graph with $\m$ at the center. If $\dim_F\m=1$, then $\G(R)$ is a single vertex and we are done.

Case 2: $\m=xR$ is a principal ideal. Let $\a$ be a non-zero ideal of $R$. Then $ \a\subseteq\m$. If $y\in\a$, then $y=rx$ for some $r \in R$. If $r$ is a unit, then $x=yr^{-1} \in \a$ and hence $ \a= \m $. If $\a\neq\m$, then $r$ is not unit and so $r=sx$, for some $s\in R$. Thus $y=sx^2$ and subsequently $\a\subseteq\m^2\subseteq\m $. Since $\G(R)$ is triangle free, it follows that $\a= \m^2$. Therefore $\G(R)$ is either a single vertex when $\m=\m^2$ or it is an edge when $\m\neq\m^2$.
\end{proof}

The following corollary is a direct consequence of the preceding two theorems.
\begin{corollary}
Let $R$ be a ring, which is neither a direct sum of two fields nor a direct sum of a field with a local ring $(S,\m)$ such that $\m$ is a field. Then the following conditions are equivalent:
\begin{itemize}
\item[(1)]$\G(R)$ is triangle-free,
\item[(2)]$\G(R)$ has a pendant,
\item[(3)]$\G(R)$ is bipartite.
\item[(4)]$\G(R)$ is star.
\end{itemize}
\end{corollary}

In what follows, we shall concentrate on cycle structure of intersection graphs and give a characterization of almost all intersection graphs under investigation that do not have an induced cycle of length greater than $3$.
\begin{theorem}
The graph $\G(R)$ is $C_{4}$-free if and only if $R$ has no set of four non-zero independent ideals.
\end{theorem}
\begin{proof}
First suppose that $R$ has an ideal which is a direct sum of four non-zero ideals, namely $\a_1,\a_2,\a_3$ and $\a_4$. Then $\a_1\oplus\a_2$, $\a_2\oplus\a_3$, $\a_3\oplus\a_4 ,\a_4\oplus\a_1$ induces a cycle of length $4$ in $\G(R)$.

Conversely, suppose that $R$ has an induced $4$-cycle with vertices $\a_1,\a_2,\a_3$ and $\a_4$. Then $\a_1\cap\a_3=\a_2\cap\a_4=0$. Since $\a_2\cap\a_3+\a_3\cap\a_4\subseteq\a_3$, we have
\begin{align*}
(\a_1\cap\a_2)\cap(\a_2\cap\a_3+ \a_3\cap\a_4+\a_4\cap\a_1)&\subseteq(\a_1\cap\a_2)\cap(\a_3+(\a_4\cap\a_1))\\
&=(\a_1\cap\a_2)\cap (\a_3\oplus\a_4\cap\a_1).
\end{align*}
If $a+b\in (\a_1\cap\a_2)\cap(\a_3\oplus\a_4\cap\a_1)$, where $a\in\a_3$ and $b\in\a_4\cap\a_1$, then $a\in\a_1$, which implies that $a=0$. Then $b\in\a_2$ and similarly $b=0$. Hence 
\[(\a_1\cap\a_2)\cap(\a_2\cap\a_3+\a_3\cap\a_4+\a_4\cap\a_1)=0.\]
Similar arguments show that $(\a_1\cap\a_2),(\a_2\cap\a_3),(\a_3\cap\a_4)$ and $(\a_4\cap\a_1)$ are non-zero independent ideals and the proof is complete.
\end{proof}

Recall that a ring is reduced if it has no non-zero nilpotent element.
\begin{theorem}
Let $R$ be a reduced ring. Then $\G(R)$ is $C_n$-free ($n\geq5$) if and only if $R$ has no set of $n$ independent of ideals.
\end{theorem}
\begin{proof}
First suppose $\G(R)$ is $C_n$-free. If $R$ has $n$ non-zero independent ideals $\a_1,\ldots\a_n$, then $\a_1\oplus\a_2,\a_2\oplus\a_3,\ldots,\a_n\oplus\a_1$ induces a cycle of length $n$ in $\G(R)$, which is a contradiction.

Now, suppose that $R$ has no set of $n$ non-zero independent ideals and the ideals $\a_1,\ldots,\a_n$ induce a cycle of length $n$. Let $\b_n=\a_n\cap\a_1$ and $\b_i =\a_i\cap\a_{i+1}$ for all $i=1,\ldots,n-1$. Then for all distinct $1\leq i,j\leq n$, we have $\b_i\b_j= 0$. Let $\b_i^*=\b_1+\cdots+\b_{i-1}+\b_{i+1}+\cdots+\b_n$. Then $\b_i\b_i^*=0$. Thus $(\b_i\cap\b_i^*)^2\subseteq\b_i\b_i^*=0$, for all $i=1,\ldots,n$. Since $R$ is reduced, it follows that  $\b_i\cap\b_i^*=0$, from which it follows that $\{\b_1,\ldots,\b_n\}$ is a set of non-zero independent ideals of $R$, which is a contradiction.
\end{proof}

In the sequel, we give another approaches to induced cycles in intersection graphs. The following lemma is straightforward.
\begin{lemma}
Suppose $\a_1,\ldots,\a_n$ induce a cycle of length $n$ in $\G(R)$. Then there exist $t$ independent ideals $\a_{i_1},\ldots,\a_{i_t}$ such that $2\leq t\leq\lfloor\frac{n}{2}\rfloor$ and $\a_{i_1}\oplus\cdots\oplus\a_{i_t}$ is adjacent to $\a_i$ for all $i=1,\ldots,n$.
\end{lemma}
\begin{theorem}
Suppose $\a_1,\ldots,\a_n$ induce a cycle of length $n$ in $\G(R)$ and the number $t$ introduced in the previous lemma takes it maximum value $\lfloor\frac{n}{2}\rfloor$. Then $R$ has a set of $n$ non-zero independent ideals if $n$ is even and it has a set of $n-1$ non-zero independent ideals if $n$ is odd.
\end{theorem}
\begin{proof}
Without loss of generality we may assume that $\a_1,\a_3,\ldots,\a_{2\lfloor\frac{n}2\rfloor-1}$ are independent. A simple verification shows that
\[\{\a_1\cap\a_2,\a_2\cap\a_3,\ldots,\a_{n-1}\cap\a_n,\a_n\cap\a_1\}\]
when $n$ is even,
\[\{\a_1\cap\a_2,\a_2\cap\a_3,\ldots,\a_{2\lfloor\frac{n}2\rfloor-1}\cap\a_{n -1},\a_n\cap\a_1\}\]
when $n$ is odd are sets of non-zero independent ideas of $R$, as required.
\end{proof}
\begin{theorem}
Suppose $\a_1,\ldots,\a_n$ $(n\geq3)$ are independent ideals of $R$. Let $\b_i=\a_{i_1}\oplus\cdots\oplus\a_{i_{n_i}}$, for $i=1,\ldots,n$. Then $\b_1,\ldots,\b_n$ induce a cycle of length $n$ if and only if there exist a permutation $\pi\in S_n$ such that $\b_i=\a_{\pi(i)}\oplus\a_{\pi(i+1)}$.
\end{theorem}
\begin{proof}
If $n=3$ then the result is obvious. If there exist $\pi \in S_n$ such that $\b_i=\a_{\pi(i)}\oplus\a_{\pi(i +1)}$, for all $i=1,\ldots,n$, then there is nothing to prove. Hence we may assume that $\b_1,\ldots,\b_n$ are vertices of an induced cycle with length $n\geq4$. Then $n_i\geq 2$, for all $i=1,\ldots,n$, otherwise $\b_j=\a_{j_1}$ for some $j$. But then $\a_{j_1}$ is adjacent to $\b_{j-1}$ and $\b_{j+1}$, which implies that $\b_{j-1}$ and $\b_{j+1}$ are adjacent, a contradiction. Hence $2n\leq \sum_{i=1}^{n}n_i$. On the other hand, the number of $\b_j$ containing $\a_i$ is at most two for all $i=1,\ldots,n$, which implies that $\sum_{i=1}^{n}n_i\leq2n$. Therefore $\sum_{i=1}^{n}n_i=2n$ and hence $n_i=2$, for all $i=1,\ldots,n$. Now the result is straightforward. 
\end{proof}

Utilizing the same method used before, we may prove the following result for $n$-claws instead of $n$-cycles.
\begin{theorem}
Let $R$ be a reduced ring. Then the ideals $\a_1,\ldots,\a_n$ of $R$ are independent and $\a_1\oplus\cdots\oplus\a_n$ is a proper ideal of $R$ if and only if there exist an induced $n$-claw in $\G(R)$.	
\end{theorem}
\begin{proof}
If $\a_1,\ldots,\a_n$ are independent ideals of $R$ such that $\a_1\oplus\cdots\oplus\a_n$ is a proper ideal of $R$, then clearly $\{\a_1,\ldots,\a_n,\a_1\oplus\cdots\oplus\a_n\}$ induces an $n$-claw in $\G(R)$.

Now, suppose that the ideals $\a_1,\ldots,\a_n$ and $\a$ are pendants and the center of an induced $n$-claw, respectively. Let 
\[\a_i^*=\a_1+\cdots+\a_{i-1}+\a_{i +1}+\cdots+\a_n,\]
for all $i=1,\ldots,n$. Then
\[(\a_i\cap\a_i^*)^2\subseteq\a_i\a_i^*=\sum_{j \neq i}\a_i\a_j\subseteq\sum_{j\neq i}\a_i\cap\a_j=0,\]
for all $i=1,\ldots,n$, which implies that $\a_1,\ldots,\a_n$ are independent. If  $R\neq\a_1\oplus\cdots\oplus\a_n$, then we are done. Now, suppose that $R=\a_1\oplus\cdots\oplus\a_n$. If $\a_i$ is not a field for some $1\leq i\leq n$, then by replacing $\a_i$ by one of its non-zero proper ideals, we may assume that $R\neq\a_1\oplus\cdots\oplus\a_n$, as required. Otherwise $\a_1,\ldots,\a_n$ are all fields. But then $\a_i\subseteq\a$, for all $i=1,\ldots,n$, which implies that $\a=R$, a contradiction.
\end{proof}
\section{Hamilton cycles}
The aim of this section is to show that except few cases all intersection graphs are Hamiltonian. Indeed, we shall prove the stronger result that such graphs are pancyclic.

A simple verification shows that if $R=S\oplus F$, where $F$ is a field and $\G(S)$ has a Hamiltonian path, then $\G(R)$ has a Hamiltonian cycle. This fact enables us to prove the following result. In what follows, the set of all ideals of a ring $R$ is denoted by $\I(R)$.
\begin{theorem}\label{hamiltonian}
Let $R$ be an Artin ring. Then $\G(R)$ is Hamiltonian if and only if $R$ is not isomorphic to the following rings:
\begin{itemize}
\item[(1)]$F$ or $E\oplus F$,
\item[(2)]$S$ or $E\oplus S$ such that $(S,F)$ is a local ring,
\item[(3)]$S$ such that $(S,T)$ is a local ring and $(T,F)$ is a local ring,
\end{itemize}
where $E$ and $F$ are fields.
\end{theorem}
\begin{proof}
If $R$ is isomorphic to one of the rings in parts (1), (2) or (3), then clearly $\G(R)$ is not Hamiltonian. Now, suppose that $R$ is a ring such that $\G(R)$ is not Hamiltonian. We proceed in some steps:

Case 1: $R=R_1\oplus R_2$ such that $|\I(R_1)|,|\I(R_2)|\geq4$. Let 
\[\I(R_1)=\{0=\a_0,\a_1,\ldots,\a_m=R_1\}\]
and
\[\I(R_2)=\{0=\b_0,\b_1,\ldots,\b_n=R_2\}.\]
Clearly an arbitrary ideal of $R$ can be expressed as $\a_i\oplus\b_j$ for some $1\leq i\leq m$ and $1\leq j\leq n$. Consider an $(m\times n)$-grid and put $\a_i\oplus\b_j$ on the $(i,j)$-th coordinate. By Figures 1, 2 and 3, the subgraph induced by ideals $\a_i\oplus\b_j$ in which $\a_i,\b_j\neq0$ is Hamiltonian with a Hamiltonian cycle in which there exists at least one edge on every row as well as one edge on every column. If $\{\a_i\oplus\b_j,\a_{i+1}\oplus\b_j\}$ is an edge such that $i,j>0$, then by removing this edge and adding two edges $\{\a_i\oplus\b_j,\b_j\}$ and $\{\b_j,\a_{i+1}\oplus\b_j\}$ we reach to a new cycle including the vertex $\b_j$. Similarly, we may enlarge the resulting cycle in which the new cyclic contains an arbitrary $\a_i\neq0$. Continuing this way, we reach to a Hamiltonian cycle for $\G(R)$, a contradiction.

Case 2: $R=R_1\oplus R_2$ such that $\I(R_1)\geq3$ and $|\I(R_2)|=3$. The same as in case 1, we may present ideals of $R$ on the grids as it is shown in Figures 4 and 5, which gives rise to a Hamiltonian cycle for $\G(R)$. Hence $\G(R)$ is Hamiltonian, which is a contradiction.

Case 3: $R=S\oplus F$, where $F$ is a field. If $S=S_1\oplus S_2$, where either $S_1$ or $S_2$, say $S_1$ is not a field, then $R=S_1\oplus(S_2\oplus F)$ and by case 1, $\G(R)$ is Hamiltonian. Now, suppose that $S_1$ and $S_2$ are both fields. Then 
\[S_1\sim S_1\oplus S_2\sim S_2\sim S_2\oplus F\sim F\sim S_1\oplus F\sim S_1\]
is a Hamiltonian cycle for $\G(R)$. Hence $\G(R)$ is Hamiltonian, a contradiction.

Case 4: If $R$ is a field or it is a direct sum of two fields, then we are done. If not, by cases 1, 2 and 3, there exists a sequence $\{(S_i,R_i)\}_{i=1}^n$ of local rings and a sequence $\{F_i\}_{i=1}^n$ of fields such that $R=R_0=S_1$ or $S_1\oplus F_1$ and $R_i=S_{i+1}$ or $S_{i+1}\oplus F_{i+1}$ for all $1\leq i<n$. Moreover, $R_n$ is a field. If $n=1$, then either $\G(R)$ is a single vertex or it is a path of length three. If $n=2$, then $R=S_1, R_1=S_2$ and $\G(R)$ is an edge. If $n\geq3$, then since $\G(R_{n-2})$ is a path, $\G(R_{n-3})$ and hence $\G(R)$ is Hamiltonian, which is a contradiction. The proof is complete.
\end{proof}
\begin{center}
\begin{tikzpicture}[scale=0.6,rotate=90]
\draw [dotted] (0,0) grid (4,4);
\draw [dotted] (5,0) grid (9,4);
\draw [dotted] (0,5) grid (4,9);
\draw [dotted] (5,5) grid (9,9);

\draw [color=white] (0,8)--(0,9)--(1,9);
\draw [color=white] (8,0)--(9,0)--(9,1);

\draw [thick] (3,4)--(3,0)--(1,0)--(1,1)--(2,1)--(2,2)--(1,2)--(1,3)--(2,3)--(2,4)--(1,4);
\draw [thick] (2,5)--(1,5)--(1,6)--(2,6)--(2,7)--(1,7)--(1,8)--(2,8);
\draw [thick] (3,5)--(3,8)--(4,8)--(4,5);
\draw [thick] (4,4)--(4,0);

\draw [thick] (5,0)--(5,4);
\draw [thick] (6,4)--(6,0)--(7,0)--(7,4);
\draw [thick] (8,2)--(8,4);
\draw [thick] (8,0)--(8,1)--(9,1)--(9,4);

\draw [thick] (5,5)--(5,8)--(6,8)--(6,5);
\draw [thick] (7,5)--(7,8)--(8,8)--(8,5);
\draw [thick] (9,5)--(9,8);

\draw [thick] (8,0) to [out=77, in=-77] (8,2);
\draw [thick] (2,8) to [out=13, in=167] (9,8);

\draw [loosely dotted,thick] (4.2,0.5)--(4.8,0.5);
\draw [loosely dotted,thick] (4.2,1.5)--(4.8,1.5);
\draw [loosely dotted,thick] (4.2,2.5)--(4.8,2.5);
\draw [loosely dotted,thick] (4.2,3.5)--(4.8,3.5);

\draw [loosely dotted,thick] (4.2,5.5)--(4.8,5.5);
\draw [loosely dotted,thick] (4.2,6.5)--(4.8,6.5);
\draw [loosely dotted,thick] (4.2,7.5)--(4.8,7.5);
\draw [loosely dotted,thick] (4.2,8.5)--(4.8,8.5);

\draw [loosely dotted,thick] (0.5,4.2)--(0.5,4.8);
\draw [loosely dotted,thick] (1.5,4.2)--(1.5,4.8);
\draw [loosely dotted,thick] (2.5,4.2)--(2.5,4.8);
\draw [loosely dotted,thick] (3.5,4.2)--(3.5,4.8);

\draw [loosely dotted,thick] (5.5,4.2)--(5.5,4.8);
\draw [loosely dotted,thick] (6.5,4.2)--(6.5,4.8);
\draw [loosely dotted,thick] (7.5,4.2)--(7.5,4.8);
\draw [loosely dotted,thick] (8.5,4.2)--(8.5,4.8);

\draw [loosely dotted,thick] (4.2,4.2)--(4.82,4.82);
\draw [loosely dotted,thick] (4.82,4.2)--(4.2,4.82);
\end{tikzpicture}\\
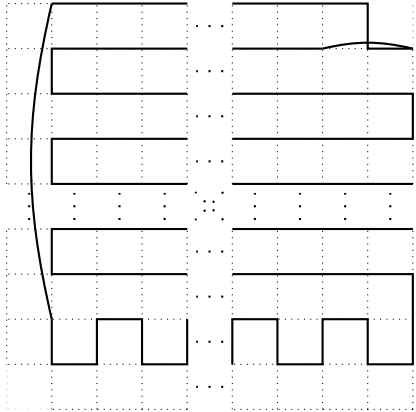
Figure 1. $(|\I(R_1)|,|\I(R_2)|)=(\mbox{odd}>3,\mbox{even}>3)$
\end{center}
\begin{center}
\begin{tikzpicture}[scale=0.6,rotate=90]
\draw [dotted] (0,0) grid (4,4);
\draw [dotted] (5,0) grid (9,4);
\draw [dotted] (0,5) grid (4,9);
\draw [dotted] (5,5) grid (9,9);

\draw [color=white] (0,8)--(0,9)--(1,9);
\draw [color=white] (8,0)--(9,0)--(9,1);

\draw [thick] (4,0)--(4,4);
\draw [thick](3,4)--(3,0)--(1,0)--(1,1)--(2,1)--(2,2)--(1,2)--(1,3)--(2,3)--(2,4)--(1,4);

\draw [thick] (2,8)--(1,8)--(1,7)--(2,7)--(2,6)--(1,6)--(1,5)--(2,5);
\draw [thick] (3,5)--(3,8)--(4,8)--(4,5);

\draw [thick] (5,4)--(5,0)--(6,0)--(6,4);
\draw [thick] (7,4)--(7,0)--(8,0);
\draw [thick] (8,4)--(8,1)--(9,1)--(9,4);

\draw [thick] (5,8)--(5,5);
\draw [thick] (6,5)--(6,8)--(7,8)--(7,5);
\draw [thick] (8,8)--(8,5);
\draw [thick] (9,8)--(9,5);

\draw [thick] (2,8) to [out=13, in=167] (9,8);
\draw [thick] (2,8) to [out=13, in=167] (9,8);
\draw [thick] (8,0) to [out=77, in=-77] (8,8);

\draw [loosely dotted,thick] (4.2,0.5)--(4.8,0.5);
\draw [loosely dotted,thick] (4.2,1.5)--(4.8,1.5);
\draw [loosely dotted,thick] (4.2,2.5)--(4.8,2.5);
\draw [loosely dotted,thick] (4.2,3.5)--(4.8,3.5);

\draw [loosely dotted,thick] (4.2,5.5)--(4.8,5.5);
\draw [loosely dotted,thick] (4.2,6.5)--(4.8,6.5);
\draw [loosely dotted,thick] (4.2,7.5)--(4.8,7.5);
\draw [loosely dotted,thick] (4.2,8.5)--(4.8,8.5);

\draw [loosely dotted,thick] (0.5,4.2)--(0.5,4.8);
\draw [loosely dotted,thick] (1.5,4.2)--(1.5,4.8);
\draw [loosely dotted,thick] (2.5,4.2)--(2.5,4.8);
\draw [loosely dotted,thick] (3.5,4.2)--(3.5,4.8);

\draw [loosely dotted,thick] (5.5,4.2)--(5.5,4.8);
\draw [loosely dotted,thick] (6.5,4.2)--(6.5,4.8);
\draw [loosely dotted,thick] (7.5,4.2)--(7.5,4.8);
\draw [loosely dotted,thick] (8.5,4.2)--(8.5,4.8);

\draw [loosely dotted,thick] (4.2,4.2)--(4.82,4.82);
\draw [loosely dotted,thick] (4.82,4.2)--(4.2,4.82);
\end{tikzpicture}\\
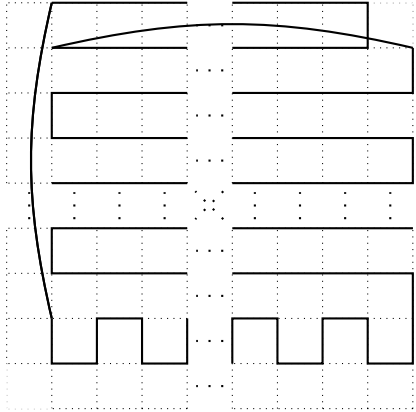
Figure 2. $(|\I(R_1)|,|\I(R_2)|)=(\mbox{odd}>3,\mbox{odd}>3)$
\end{center}
\begin{center}
\begin{tikzpicture}[scale=0.6,rotate=90]
\draw [dotted] (0,0) grid (4,4);
\draw [dotted] (5,0) grid (9,4);
\draw [dotted] (0,5) grid (4,9);
\draw [dotted] (5,5) grid (9,9);

\draw [color=white] (0,8)--(0,9)--(1,9);
\draw [color=white] (8,0)--(9,0)--(9,1);

\draw [thick] (1,0)--(2,0)--(2,1)--(1,1)--(1,2)--(2,2)--(2,3)--(1,3)--(1,4)--(2,4);
\draw [thick] (2,5)--(1,5)--(1,6)--(2,6)--(2,7)--(1,7)--(1,8)--(2,8);
\draw [thick] (3,5)--(3,8)--(4,8)--(4,5);
\draw [thick] (3,4)--(3,0);
\draw [thick] (4,4)--(4,0);

\draw [thick] (5,0)--(5,4);
\draw [thick] (6,4)--(6,0)--(7,0)--(7,4);
\draw [thick] (8,2)--(8,4);
\draw [thick] (8,0)--(8,1)--(9,1)--(9,4);

\draw [thick] (5,5)--(5,8)--(6,8)--(6,5);
\draw [thick] (7,5)--(7,8)--(8,8)--(8,5);
\draw [thick] (9,5)--(9,8);

\draw [thick] (8,0) to [out=77, in=-77] (8,2);
\draw [thick] (2,8) to [out=13, in=167] (9,8);
\draw [thick] (1,0) to [out=13, in=167] (3,0);

\draw [loosely dotted,thick] (4.2,0.5)--(4.8,0.5);
\draw [loosely dotted,thick] (4.2,1.5)--(4.8,1.5);
\draw [loosely dotted,thick] (4.2,2.5)--(4.8,2.5);
\draw [loosely dotted,thick] (4.2,3.5)--(4.8,3.5);

\draw [loosely dotted,thick] (4.2,5.5)--(4.8,5.5);
\draw [loosely dotted,thick] (4.2,6.5)--(4.8,6.5);
\draw [loosely dotted,thick] (4.2,7.5)--(4.8,7.5);
\draw [loosely dotted,thick] (4.2,8.5)--(4.8,8.5);

\draw [loosely dotted,thick] (0.5,4.2)--(0.5,4.8);
\draw [loosely dotted,thick] (1.5,4.2)--(1.5,4.8);
\draw [loosely dotted,thick] (2.5,4.2)--(2.5,4.8);
\draw [loosely dotted,thick] (3.5,4.2)--(3.5,4.8);

\draw [loosely dotted,thick] (5.5,4.2)--(5.5,4.8);
\draw [loosely dotted,thick] (6.5,4.2)--(6.5,4.8);
\draw [loosely dotted,thick] (7.5,4.2)--(7.5,4.8);
\draw [loosely dotted,thick] (8.5,4.2)--(8.5,4.8);

\draw [loosely dotted,thick] (4.2,4.2)--(4.82,4.82);
\draw [loosely dotted,thick] (4.82,4.2)--(4.2,4.82);
\end{tikzpicture}\\
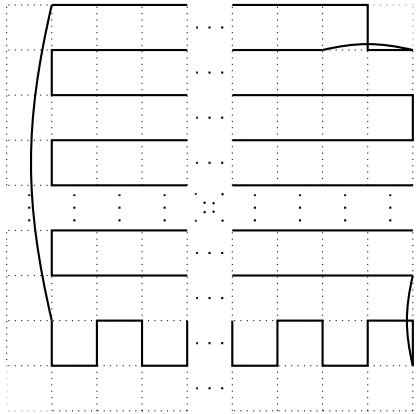
Figure 3. $(|\I(R_1)|,|\I(R_2)|)=(\mbox{even}>3,\mbox{even}>3)$
\end{center}
\begin{center}
\begin{tikzpicture}[scale=0.6]

\draw [dotted] (0,0) grid (4,2);
\draw [dotted] (5,0) grid (9,2);
\draw [color=white] (0,1)--(0,0)--(1,0);
\draw [color=white] (8,2)--(9,2)--(9,1);

\draw [thick] (0,1)--(0,2)--(1,2)--(1,1)--(2,1)--(2,2)--(3,2)--(3,1)--(4,1)--(4,2);
\draw [thick] (5,2)--(5,1)--(6,1)--(6,2)--(7,2)--(7,1)--(8,1)--(8,2);
\draw [thick] (8,0)--(9,0)--(9,1);

\draw [thick] (8,2) to [out=-60, in=60] (8,0);
\draw [thick] (0,1) to [out=-13, in=-167] (9,1);

\draw [loosely dotted,thick] (4.2,.5)--(4.82,.5);
\draw [loosely dotted,thick] (4.2,1.5)--(4.82,1.5);

\end{tikzpicture}\\
Figure 4. $(|\I(R_1)|,|\I(R_2)|)=(3,\mbox{even}\geq3)$
\end{center}
\begin{center}
\begin{tikzpicture}[scale=0.6]

\draw [dotted] (0,0) grid (4,2);
\draw [dotted] (5,0) grid (9,2);
\draw [color=white] (0,1)--(0,0)--(1,0);
\draw [color=white] (8,2)--(9,2)--(9,1);

\draw [thick] (0,1)--(0,2)--(1,2)--(1,1)--(2,1)--(2,2)--(3,2)--(3,1)--(4,1)--(4,2);
\draw [thick] (5,1)--(5,2)--(6,2)--(6,1)--(7,1)--(7,2)--(8,2)--(8,0);
\draw [thick] (8,0)--(9,0)--(9,1);

\draw [thick] (0,1) to [out=-13, in=-167] (9,1);
\draw [thick , color=white] (0,0)--(1,0);

\draw [loosely dotted,thick] (4.2,.5)--(4.82,.5);
\draw [loosely dotted,thick] (4.2,1.5)--(4.82,1.5);
\end{tikzpicture}\\
Figure 5. $(|\I(R_1)|,|\I(R_2)|)=(3,\mbox{odd}\geq3)$
\end{center}
\begin{theorem}\label{pancyclic}
Let $R$ be an Artin ring. Then $\G(R)$ is Hamiltonian if and only if it is pancyclic.
\end{theorem}
\begin{proof}
If $\G(R)$ is pancyclic, then clearly $\G(R)$ is Hamiltonian. Now, we show that the converse is also true. Suppose on the contrary that there is an Artin ring $R$ such that $\G(R)$ is a non-pancyclic Hamiltonian graph and that $R$ is minimal with this property. If $R$ is neither a local ring nor a direct sum of a local ring with a field, then by applying the following transformations on the Hamiltonian cycles constructed in cases 1 and 2 of Theorem \ref{hamiltonian}, along with replacing horizontal or vertical paths of length two to a path of length one, by joining its end vertices, we would reach to cycles with possible arbitrary length $\geq4$. 
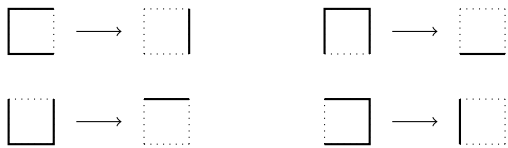
\begin{center}
\begin{tikzpicture}[scale=0.6]
\draw [dotted] (0,0) grid (1,1);
\draw [dotted] (3,0) grid (4,1);

\draw [dotted] (7,0) grid (8,1);
\draw [dotted] (10,0) grid (11,1);

\draw [dotted] (0,2) grid (1,3);
\draw [dotted] (3,2) grid (4,3);

\draw [dotted] (7,2) grid (8,3);
\draw [dotted] (10,2) grid (11,3);

\draw [thick] (0,1)--(0,0)--(1,0)--(1,1);
\draw [->] (1.5,0.5)--(2.5,0.5);
\draw [thick] (3,1)--(4,1);

\draw [thick] (7,1)--(8,1)--(8,0)--(7,0);
\draw [->] (8.5,0.5)--(9.5,0.5);
\draw [thick] (10,0)--(10,1);

\draw [thick] (1,2)--(0,2)--(0,3)--(1,3);
\draw [->] (1.5,2.5)--(2.5,2.5);
\draw [thick] (4,2)--(4,3);

\draw [thick] (7,2)--(7,3)--(8,3)--(8,2);
\draw [->] (8.5,2.5)--(9.5,2.5);
\draw [thick] (10,2)--(11,2);
\end{tikzpicture}
\end{center}
On the other hand, by Theorem \ref{triangle}, the graphs under consideration contain triangles, which implies that $\G(R)$ is pancyclic, a contradiction. Hence either $R=S$ or $R=S\times F$, where $(S,\m)$ is a local ring and $F$ is a field. If $\G(\m)$ is Hamiltonian, then either $\G(\m)$ is pancyclic, which implies that $\G(R)$ is pancyclic too, contradicting the hypothesis, or $\G(\m)$ is not pancyclic which contradicts the minimality of $R$. Thus $\G(\m)$ is not Hamiltonian and $\m$ is isomorphic to one of the five rings given in Theorem \ref{hamiltonian}. Now, a simple verification shows that in each case either $\G(\m)$ is not Hamiltonian or it is pancyclic, which is a contradiction. The proof is complete.
\end{proof}


\begin{thebibliography}{0}
\bibitem{jb}J. Bosak, The graphs of semigroups, \textit{Theory of Graphs and Application}, Academic Press, New York, 1964, 119--125.
\bibitem{ic-sg-tkm-mks}I. Chakrabarty, S. Ghosh, T. K. Mukherjee and M. K. Sen, Intersection graphs of ideals of rings, \textit{Discrete Math.} \textbf{309} (2009), 5381–-5392.
\bibitem{bc-gp}B. Cs\'{a}k\'{e}any and G. Poll\'{a}k, The graph of subgroups of a finite group, \textit{Czech. Math. J.} \textbf{19} (1969), 241--247.
\bibitem{mh-pl-mm}M. Herzog, P. Longobardi and M. Maj, On a graph related to the maximal subgroups of a group, \textit{Bull. Aust. Math. Soc.} \textbf{81} (2010), 317--328.
\bibitem{shj-njr-1}S. H. Jafari and N. Jafari Rad, Domination in intersection graphs of rings and modules, \textit{Ital. J. Pure Appl. Math.} \textbf{28} (2011), 17--20.
\bibitem{shj-njr-2}S. H. Jafari and N. Jafari Rad, Planarity of intersection graphs of ideals of rings, \textit{Int. Electron. J. Algebra} \textbf{8} (2010), 161--166.
\bibitem{brm}B. R. Macdonald, \textit{Finite Rings with Identity}, Marcel Dekker, Inc., New York, 1974.
\bibitem{tam-frm}T. A. McKee and F. R. McMorris, \textit{Topic in Intersection Graph Theory}, SIAM, Philadelphia, 1999.
\bibitem{rs}R. Shen, Intersection graphs of subgroups of finite groups, \textit{Czech. Math. J.} \textbf{135} (2010), 945--950.
\bibitem{bz}B. Zelinka, Intersection graphs of finite abelian groups, \textit{Czech. Math. J.} \textbf{25} (1975),
171--174.
\end{thebibliography}
\end{document}